\documentclass[14pt,a4paper]{extarticle}
\usepackage[T1,T2A]{fontenc} 
\usepackage[utf8]{inputenc} 
\usepackage[russian, english]{babel} 
\usepackage{enumitem}

\usepackage{geometry}
\usepackage{amsmath,amssymb,amsfonts,amsthm,mathrsfs}
\usepackage{bbm}  % for \mathbbm 1 or ind
\usepackage{bbding}
\usepackage{graphicx,graphics,bm,array}

\usepackage{hyperref} % для ссылок

\numberwithin{equation}{section}

\theoremstyle{plain}
\newtheorem{thm}{Theorem}[section]
\newtheorem{prop}{Proposition}[section]

\newtheorem{lem}{Lemma}[section]

\theoremstyle{definition}

\theoremstyle{remark}
\newtheorem{rem}{Remark}

\newcommand{\cB}{\mathscr B}
\newcommand{\cF}{\mathscr F}

\newcommand{\cH}{\mathscr H}

\newcommand{\cT}{\mathcal T}
\newcommand{\RR}{\mathbb R}
\newcommand{\RRl}{\overline{\mathbb R}}
\newcommand{\QQ}{\mathbb Q}

\newcommand{\FF}{\mathbb F}
\newcommand{\PP}{\sf P}
\newcommand{\EE}{\sf E}
\newcommand{\ind}{{\mathbbm 1}}

\title{Explicit Predictable Compensators for Single Jump Processes with Initial Information}
\author{
    Assylliya K. Zhunussova\\
\small $^{1}$Lomonosov Moscow State University, Moscow, Russia \\
    \small \href{mailto:lichka\_archive@mail.ru}{lichka\_archive@mail.ru}\,\Envelope, \url{https://orcid.org/0000-0002-8801-8381}}
\date{}
\begin{document}
\maketitle

\thispagestyle{plain} 


\begin{abstract}
    We study the predictable compensators of stochastic processes in a single jump filtration augmented with initial information represented by a sub-$\sigma$-algebra $\cH$. We consider adapted càdlàg processes of finite variation and give an explicit construction of their predictable compensators. The main difficulty arises when the jump size has a heavy tail and lacks integrability, so that the classical Doob--Meyer decomposition does not apply. To overcome this, we use the theory of $\sigma$-martingales. We establish necessary and sufficient conditions for a process to be a $\sigma$-martingale and explicitly compute the compensator of a suitably weighted process. This yields an explicit relation between the continuous drift and the expected jump component of the original process.
\end{abstract}
\section{Introduction}
In the general theory of stochastic processes, the construction of predictable compensators for single-jump processes is a classical and fundamental problem. 
Historically, this problem has been exhaustively studied within the framework of a trivial initial $\sigma$-algebra, pioneered by the French school \cite{ChouMeyer1975,Dellacherie1970,DellacherieMeyer1982} and further developed in modern literature \cite{HerdegenHerrmann2016}. 
However, advanced applications often require a more generalized framework utilizing a single jump filtration augmented by a non-trivial initial $\sigma$-algebra $\cH$, as introduced by Gushchin \cite{Gushchin2020}. 
Within such an extended filtration, any adapted, right-continuous stochastic process can be canonically represented via an $\cH$-measurable pre-jump trajectory and a terminal jump variable \cite{Zhunussova2026}. 
A critical analytical challenge arises when the jump size possesses a heavy tail and fails to be integrable. 
Under this condition, the process is no longer a special semimartingale \cite{JacodShiryaev2003}, and the classical Doob-Meyer decomposition cannot be applied directly to extract its predictable compensator. 

To overcome this integrability bottleneck, we appeal to the broader class of $\sigma$-martingales, systematically formalized by Jacod and Shiryaev \cite{JacodShiryaev2003}.
By introducing a strictly positive predictable damping multiplier, $\sigma$-martingales allow for the localization of processes with heavy-tailed jumps, preserving their underlying martingale structure. 
Furthermore, the explicit computation of predictable compensators within this framework inherently relies on the behavior of the conditional survival probability. 
To rigorously ensure the correctness of our integral representations and to avoid division by zero, we invoke the fundamental properties of the Azéma supermartingale. 
Specifically, we leverage Jeulin's classical result \cite{Jeulin1980}, which guarantees the strict positivity of the left-hand limit of the survival process strictly prior to the jump.

\subsection{Preliminaries and Base Criteria}\label{subsec: preliminaries and base criteria}

Let $(\Omega,\cF,\PP)$ be a fixed probability space equipped with a sub-$\sigma$-algebra $\cH\subset\cF$, representing the initial information available at $t=0$. 
Let $\gamma$ be a random time taking values in the extended half-line $\RRl_+=[0,\infty]$, representing the jump time. 
Throughout this paper, we operate under the standing assumption that $\PP(\gamma>0)=1$, which strictly precludes instantaneous jumps at the origin. 
We define the regular conditional distribution function of the jump time given the initial information $\cH$ as $G_{\cH}(t)=\PP(\gamma\le t\mid\cH)$, and the corresponding conditional survival probability as $\overline{G}{\cH}(t)=\PP(\gamma>t\mid\cH)$. 
The deterministic right endpoint of the unconditional jump distribution is denoted by $t_G=\sup\{t\in\RR_+:\PP(\gamma\le t)< 1\}$.

Relying on the structural properties established in \cite{Zhunussova2026}, we consider the single jump filtration with initial information $\FF=(\cF_t)_{t\ge0}$. 
For $t<\infty$, a set $A\in\cF_t$ if and only if there exists a set $C\in\cH$ such that $A\cap\{t<\gamma\}=C\cap\{t<\gamma\}$.
The limit $\sigma$-algebra $\cF_\infty$ is constructed to ensure that the trace of $\cF_\infty$ on the event $\{\gamma=\infty\}$ strictly coincides with the trace of $\cH$. 
Following the framework of Dellacherie and Meyer \cite{DellacherieMeyer1982}, we assume henceforth that the filtration $\FF$ is augmented with all $\PP$-null sets of $\cF$. 
As shown in \cite{Zhunussova2026}, this augmentation, combined with the right-continuity of the generated $\sigma$-algebras, ensures that the filtration $\FF$ satisfies the usual conditions. 
This topological assumption is a technical but crucial prerequisite, as it guarantees the existence and regularity of predictable projections and validates the application of the general theory of processes.

To ensure a self-contained exposition, we recall the foundational structural results from \cite{Zhunussova2026}. 
Any $\FF$-adapted, right-continuous stochastic process $X=(X_t)_{t\ge0}$ admits the canonical pathwise representation $X_t=F(t)\ind{\{t<\gamma\}}+L\ind{\{t\ge\gamma\}}$, where $F=(F_t)_{t\ge0}$ is an $\cH$-measurable pre-jump process and $L$ is an $\cF_\infty$-measurable terminal random variable.
Furthermore, the progressive measurability (and predictability) of the process $X$ is entirely characterized by its pre-jump behavior: $X$ is progressively measurable (respectively, predictable) if and only if there exists a $\cB(\RR_+)\otimes\cH$-measurable function $C(t,\omega)$ such that $X_t=C(t,\omega)$ on the set $\{t<\gamma\}$ (respectively, on $\{t\le\gamma\}$).

To explicitly compute the predictable compensators in the subsequent sections, our primary analytical tools are the martingale and local martingale criteria, formulated as follows.

\begin{thm}[Martingale and Local Martingale Criteria]\label{thm: mart and loc mart}
    Let $\cT=\{t\in\RR_+:\PP(\gamma\ge t)>0\}$. 
    Let $F=(F(t))_{t\in\cT}$ be an $\cH$-measurable process with càdlàg paths, and let $L$ be an arbitrary random variable. 
    Consider the process $M=(M_t)_{t\in\RRl_+}$ defined by the equality $M_t=F(t)\ind{\{t<\gamma\}}+L\ind{\{t\ge\gamma\}}$.
    \begin{enumerate}[label=(\roman*)]
        \item\label{thm: mart and local mart: mart} The process $(M_t)_{t\in\cT}$ is a martingale if and only if for all $t\in\cT$, the integrability condition $\EE[|M_t|]<\infty$ holds, and $\EE[M_t-M_0\mid\cH]=0$ almost surely.
        \item\label{thm: mart and local mart: local mart} Assume further that $\EE[\sup_{t\in[0,t_G)}|F(t)|]<\infty$ and $\EE[|L|]<\infty$. The process $M$ is a local martingale on $\RRl_+$ if and only if $(M_t)_{t\in\cT}$ is a martingale, and the boundary equality $\EE[L\mid\cH]\ind_{\{\gamma=t_G\}}=F(0)\ind_{\{\gamma=t_G\}}$ holds almost surely on the set $\{\gamma=t_G\}$.
    \end{enumerate}
\end{thm}

Finally, the strict positivity of the conditional survival probability prior to the jump is essential to prevent division by zero in our integral representations. This property is a classical consequence of the theory of Azéma supermartingales. Since the original proof relies on the methods of general topology, we provide a concise probabilistic adaptation tailored specifically to our extended single jump filtration.

\begin{lem}[Adapted from Jeulin \cite{Jeulin1980}]\label{lem: Jeulin}
    On the stochastic interval $[\![0,\gamma]\!]$, the left-hand limit of the conditional survival probability is strictly positive almost surely; that is, $\overline{G}_{\cH}(t-)>0$ a.s. for all $t\le\gamma$.
\end{lem}
\begin{proof}
    We define the random time $R$:
    $$
    R = \inf\{t \ge 0 : \overline{G}_{\cH}(t-) = 0\}
    $$
    Since the process $\overline{G}_{\cH}(t-)$ is left-continuous, $R$ is a predictable stopping time.  
    
    Suppose that this time is finite, i.e., $R<\infty$. 
    By the definition of $R$, the process equals zero at this exact moment: $\overline{G}_{\cH}(R-)=0$. 
    By definition, this implies $\PP(\gamma\ge R\mid\cH)=0$. 
    The expectation of a non-negative random variable (an indicator) is zero if and only if the variable is zero almost surely. 
    Thus, we obtain $\ind_{\{\gamma\ge R\}}=0$ almost surely on the set $\{R<\infty\}$. In other words, $\gamma<R$ almost surely.
    This means that the true default time $\gamma$ strictly precedes the moment the process $\overline{G}_{\cH}(t-)$ reaches zero.  
    
    Now, consider any arbitrary time $t$ on our stochastic interval, meaning $t\le\gamma$. 
    Since we established that $\gamma<R$ almost surely, we have the strict inequality $t\le\gamma<R$, which implies $t<R$. 
    Because $R$ is the moment of the first zero-crossing of the process $\overline{G}_{\cH}(t-)$, and our time $t$ strictly precedes $R$, the process has not yet reached zero at time $t$. 
    Consequently, $\overline{G}_{\cH}(t-)>0$ almost surely for all $t\le\gamma$.
\end{proof}

\subsection{Path Regularity and Indistinguishability}

In the classical single jump filtration framework of A.A. Gushchin \cite{Gushchin2020} with a trivial initial $\sigma$-algebra, the pre-jump processes are deterministic functions. 
Under such conditions, almost sure equality trivially ensures pointwise identity. 
However, incorporating a non-trivial initial information $\cH$ intrinsically renders the pre-jump components stochastic. 
Consequently, establishing global indistinguishability from a mere modification is no longer immediate and strictly demands explicit path regularity.

\begin{prop}[Sufficient Condition for Indistinguishability]
    Let $X$ and $Y$ be $\FF$-adapted processes that are modifications of each other. 
    If their corresponding $\cH$-measurable pre-jump processes $F^X$ and $F^Y$ possess càdlàg paths, then $X$ and $Y$ are indistinguishable.
\end{prop}
\begin{proof}
    By the adaptedness criterion, the processes admit the canonical representations $X_t=F^X(t)\ind_{\{t<\gamma\}}+L^X\ind_{\{t\ge\gamma\}}$ and $Y_t=F^Y(t)\ind_{\{t<\gamma\}}+L^Y\ind_{\{t\ge\gamma\}}$, where $L^X$ and $L^Y$ are $\cF_\infty$-measurable random variables. 
    We analyze their pathwise coincidence by partitioning the temporal domain relative to the jump time into two sets: $\{t\ge\gamma\}$ and $\{t<\gamma\}$.

    On the post-jump set $\{t\ge\gamma\}$, the processes evaluate to the time-independent random variables $L^X$ and $L^Y$. 
    Since $X$ and $Y$ are modifications, $\PP(X_t=Y_t)=1$ for all $t\ge0$, which guarantees $\PP(L^X=L^Y)=1$. 
    Thus, the processes are trivially indistinguishable on this set.

    On the pre-jump set $\{t<\gamma\}$, we define the $\cH$-measurable difference process $Z(t)=F^X(t)-F^Y(t)$, which strictly inherits càdlàg paths. 
    By hypothesis, for any fixed $t\ge0$, $\PP(\{Z(t)\neq0\}\cap\{t<\gamma\})=0$. 
    Let $\QQ_+$ denote the countable dense subset of rational numbers in $\RR_+$. 
    We define the exceptional global null set $N$:
    $$
    N=\bigcup_{q\in\QQ_+}\left(\{Z(q)\neq0\}\cap\{q<\gamma\}\right).
    $$
    Since $\QQ_+$ is countable, $\PP(N)=0$. 
    For any sample path $\omega\in\Omega\setminus N$ and any real time $t<\gamma(\omega)$, we can select a strictly decreasing sequence $\{q_n\}_{n\ge1}\subset\QQ_+$ such that $q_n\downarrow t$ and $q_n<\gamma(\omega)$. 
    Since $\omega\notin N$, the identity $Z(q_n,\omega)=0$ holds for all $n$.
    Leveraging the right-continuity of the paths of $Z$, we pass to the limit:
    $$
    Z(t,\omega)=\lim\limits_{n\to\infty}Z(q_n,\omega)=0.
    $$
    Thus, outside the single null set $N$, the equality $X_t(\omega)=Y_t(\omega)$ holds simultaneously for all $t\ge0$. 
    This establishes the strict indistinguishability of the processes $X$ and $Y$ on $\RR_+$. 
\end{proof}

\subsection{Structure of the Paper}
The remainder of this paper is organized as follows. Section 2 characterizes the class of right-continuous processes of finite variation. The main computational results are presented in Section 3, where we derive the explicit predictable compensators, culminating in a strict balance equation between the continuous drift and the expected jump risk. Finally, Section 4 addresses the analytical challenge of non-integrable jump sizes with heavy tails; by introducing necessary and sufficient criteria for $\sigma$-martingales, we provide a complete localization framework that legitimizes the compensator equations under generalized conditions.
\section{Processes of Finite Variation and Semimartingales}

By the progressive measurability criterion (see \cite[Proposition 2.1]{Zhunussova2026}), any progressively measurable — and in particular, adapted right-continuous — process $X=(X_t)_{t\ge 0}$ in the filtration $\FF$ admits the canonical representation:
$$
X_t=F(t)\ind_{\{t<\gamma\}}+L\ind_{\{t\ge\gamma\}}
$$
where $F(t)$ is an $\cH$-measurable process, and $L$ is an $\cF_\infty$-measurable random variable. The following proposition establishes the exact conditions under which the paths of the process $X$ possess finite variation on compact intervals.

\begin{prop}
    Let $X$ be an adapted process with the canonical representation $X_t=F(t)\ind_{\{t<\gamma\}}+L\ind_{\{t\ge\gamma\}}$. The process X is a process of finite variation if its pre-jump process $F$ has finite variation on $[0,t]$ for every $t<t_G$ (in the case $\PP(\gamma=t_G)=0)$, and on $[0,t_G)$ (in the case $P(\gamma=t_G)>0$).
\end{prop}
\begin{proof}
    By the adaptedness criterion (see \cite[Proposition 2.2]{Zhunussova2026}), for each $t\ge0$, the random variables $X_t$ and $F(t)$ coincide almost surely on the set $\{t<\gamma\}$.
    In the case $\PP(\gamma=t_G)=0$, the claim is obvious.
    Consider the case $\PP(\gamma=t_G)>0$. On the set $\{\gamma=t_G\}$, we have $X_t=F(t)$ for all $t<t_G$. By hypothesis, $F$ has finite variation on $[0,t_G)$. Since the process $X$ is constant for $t\ge\gamma$ and the jump magnitude at $\gamma$ is finite, the total variation of $X$ is finite almost surely. 
\end{proof}

\begin{prop}[Semimartingale Criterion]\label{prop: semimar}
    Every semimartingale in the filtration $\FF$ is a process of finite variation.
\end{prop}
\begin{proof}
By definition, any semimartingale can be decomposed into a local martingale and a process of finite variation. Thus, to prove the proposition, it suffices to show that any local martingale in this filtration has paths of finite variation.

First, assume that $M=(M_t)_{t\ge0}$ is a uniformly integrable martingale. There exists an $\cF_\infty$-measurable random variable $M_\infty$ such that $\lim_{n\to\infty}M_n=M_\infty$ almost surely, and $M_t=\EE[M_\infty\mid\cF_t]$. Since $M$ is a right-continuous adapted process, it admits the canonical representation $M_t=F(t)\ind_{\{t<\gamma\}}+L\ind_{\{t\ge\gamma\}}$.

On the set $\{t<\gamma\}$, this yields the identity:
$$
\EE[M_\infty\mid\cF_t]\ind_{\{t<\gamma\}}=F(t)\ind_{\{t<\gamma\}}.
$$
Since $\ind_{\{t<\gamma\}}$ is $\cF_t$-measurable, the left-hand side can be written as $\EE[M_\infty\ind_{\{t<\gamma\}}\mid\cF_t]$. Taking the conditional expectation with respect to the initial $\sigma$-algebra $\cH$ on both sides, and noting that $\cH\subset\cF_t$, we obtain:
$$
\EE[M_\infty\ind_{\{t<\gamma\}}\mid\cH]=F(t)\overline{G}_{\cH}(t).
$$
According to Lemma \ref{lem: Jeulin}, the conditional survival probability satisfies $\overline{G}_{\cH}(t)>0$ almost surely on the set $\{t<\gamma\}$. Therefore, we can legitimately divide by this probability to express $F(t)$ as the ratio:
$$
F(t)=\frac{\EE[M_\infty\ind_{\{t<\gamma\}}\mid\cH]}{\overline{G}_{\cH}(t)}
$$
Both the numerator and the denominator are right-continuous functions of bounded variation on $[0,t_G]$. Consequently, their ratio $F(t)$ is a càdlàg function that possesses finite variation on $[0,t_G)$ in the case $\PP(\gamma=t_G)>0$, and on $[0,t]$ for every $t<t_G$ in the case $\PP(\gamma=t_G)=0$. Setting $L=M_\infty\ind_{\{\gamma<\infty\}}$, the post-jump conditions are satisfied almost surely. This proves that a uniformly integrable martingale is indistinguishable from a regular right-continuous process of finite variation.

Finally, if $M$ is a local martingale, there exists a localizing sequence of stopping times $(T_n)_{n\ge1}$ such that $T_n\uparrow\infty$ almost surely, and each stopped process $M^{T_n}$ is a uniformly integrable martingale. Since almost all paths of each $M^{T_n}$ have finite variation, it follows that almost all paths of $M$ inherently have finite variation. 
\end{proof}

\begin{rem}
    Since every semimartingale in the filtration $\FF$ is a process of finite variation (Proposition \ref{prop: semimar}), the standard Itô calculus is not required. For any semimartingale $X$ and any predictable process $H$, the stochastic integral $\int_{(0,t]}H_sdX_s$ is well-defined pathwise as a classical Lebesgue-Stieltjes integral.
\end{rem}
\section{Explicit Predictable Compensators for Processes of Finite Variation}

Let $X$ be a right-continuous adapted process of finite variation with the canonical representation $X_t=F(t)\ind_{\{t<\gamma\}}+L\ind_{\{t\ge\gamma\}}$. 
We assume that the process satisfies the strict integrability conditions $\EE[|L|]<\infty$ and $\EE[\sup_{t\in[0,t_G)}|F(t)|]<\infty$. Under these assumptions, $X$ is a special semimartingale. 
By the general theory of processes \cite{DellacherieMeyer1982,JacodShiryaev2003}, there exists a unique predictable process of finite variation $A=(A_t)_{t\ge0}$ with $A_0=0$, called the predictable compensator, such that the compensated process $M=X-A$ is a local martingale.

Since the compensator $A$ is a predictable process, by the predictability criterion (see \cite[Proposition 2.5]{Zhunussova2026}), there exists an $\cH$-measurable process $C(t)$ such that $A_t\ind_{\{t\le\gamma\}}=C(t)\ind_{\{t\le\gamma\}}$. 
Furthermore, since the underlying process $X$ inherently stops at the jump time $\gamma$ (i.e., $X_t=X_{t\wedge\gamma}$ for all $t$), its predictable compensator must also be stopped at $\gamma$. 
Thus, the compensator admits the explicit structural representation:
$$
A_t=C(t\wedge\gamma).
$$
To compute the unknown $\cH$-measurable function $C(t)$, we rely on the condition that $M_t=X_t-C(t\wedge\gamma)$ is a local martingale. Following the notation established by Gushchin \cite{Gushchin2020}, we introduce the $\cH$-measurable conditional expectation of the jump size:
$$
H(t)=\EE[L\mid\mathcal{H},\gamma=t].
$$

\begin{thm}[Explicit Form of the Predictable Compensator]\label{thm: compensator}
    Let $X$ be the special semimartingale satisfying the integrability conditions introduced above. Its unique predictable compensator is given by $A_t=C(t\wedge\gamma)$, where the $\cH$-measurable function $C(t)$ is explicitly computed as follows.
    \begin{enumerate}[label=(\roman*)]
        \item\label{thm: compensator: C(t)} For all $t<t_G$, the function $C(t)$ is given by the integral:
        \begin{equation}\label{eq: C}
            C(t)=\int_{(0, t]}\frac{d\left[F(s)\overline{G}_{\cH}(s)\right] + H(s) dG_{\cH}(s)}{\overline{G}_{\cH}(s-)}
        \end{equation}
        \item\label{thm: compensator: C(t_G)} Furthermore, if the jump distribution possesses an atom at the right endpoint, i.e., $\PP(\gamma=t_G)>0$, the exact value of the function $C$ at $t_G$ is given by $C(t_G)=C(t_G-)+\Delta C(t_G)$, where the terminal jump is:
        \begin{equation}\label{eq: C(t_G)}
            \Delta C(t_G)=H(t_G)-F(t_G-)
        \end{equation}
    \end{enumerate}
\end{thm}
\begin{proof}
Substituting the canonical representations of the process $X$ and its predictable compensator $A_t=C(t\wedge\gamma)$, the compensated process $M_t=X_t-A_t$ takes the form:
$$
M_t=\left(F(t)-C(t)\right)\ind_{\{t<\gamma\}}+\left(L-C(\gamma)\right)\ind_{\{t\ge\gamma\}}.
$$
According to the local martingale criterion (Theorem \ref{thm: mart and loc mart}.\ref{thm: mart and local mart: local mart}), $M$ is a local martingale if and only if the corresponding integrability conditions hold, the martingale equality $\EE[M_t\mid\cH]=\EE[M_0\mid\cH]$ is satisfied for all $t<t_G$, and the specific boundary condition holds on the atomic set $\{\gamma=t_G\}$. The integrability conditions are satisfied by the hypotheses on $F$ and $L$, and the bounded variation of the compensator. Thus, it suffices to verify the martingale equality on $[0,t_G)$ and subsequently resolve the atomic case at $t_G$.

For the initial value $t=0$, we have $M_0=X_0-A_0$. On the set $\{0<\gamma\}$, this evaluates to $M_0=F(0)-C(0)$. Since both $F(0)$ and $C(0)$ are $\cH$-measurable, the right-hand side of the martingale equality is strictly $F(0)-C(0)$.

For any $t<t_G$, we decompose the process $M_t$ into two components prior to taking the conditional expectation:
$$
M_t=M_t\ind_{\{t<\gamma\}}+M_t\ind_{\{t\ge\gamma\}}=\left(F(t)-C(t)\right)\ind_{\{t<\gamma\}}+\left(L-C(\gamma)\right)\ind_{\{t\ge\gamma\}}.
$$
Taking the conditional expectation of $M_t$ with respect to $\cH$ yields:
$$
\EE[M_t\mid\cH]=\left(F(t)-C(t)\right)\PP(t<\gamma\mid\cH) + \EE\left[(L-C(\gamma))\ind_{\{\gamma\le t\}}\mid\cH\right].
$$
By the properties of conditional expectation and utilizing the notation $H(s)=\EE[L\mid\cH,\gamma=s]$, the first term directly incorporates the survival probability $\overline{G}_{\cH}(t)$, while the second term can be explicitly written as an integral with respect to the conditional distribution of the jump. Equating this expression to the initial value, the martingale condition on $[0,t_G)$ takes the form of the balance equation:
$$
\left(F(t)-C(t)\right)\overline{G}_{\cH}(t)+\int_{(0, t]}\left(H(s)-C(s)\right)dG_{\cH}(s)=F(0)-C(0).
$$
By expressing the increment of the first term as an integral with respect to its differential, the balance equation can be rewritten in the integral form:
$$
\int_{(0, t]} d\big[(F(s) - C(s))\overline{G}_{\mathcal{H}}(s)\big] + \int_{(0, t]} \big(H(s) - C(s)\big) dG_{\mathcal{H}}(s) = 0
$$
Since this equality holds for any arbitrary $t<t_G$, we can equate the corresponding integrands to obtain the differential equation:
$$
d\left[(F(s)-C(s))\overline{G}_{\cH}(s)\right]+\left(H(s)-C(s)\right)dG_{\cH}(s)=0.
$$
By the linearity of the differential, we separate the terms. Rearranging the equation to isolate all terms containing the unknown function $C(s)$ on the left-hand side, we obtain:
$$
d\left[C(s)\overline{G}_{\cH}(s)\right]+C(s)dG_{\cH}(s)=d\left[F(s)\overline{G}_{\cH}(s)\right]+H(s)dG_{\cH}(s).
$$
%%%%%%%%%%%%%%%%%%%%%%%%%%%%%%%%%%%%%%%%%%%%%%%

Applying the stochastic integration by parts formula for processes of finite variation to the product $C(s)\overline{G}_{\cH}(s)$, we observe a subtle topological nuance. 
By classical theory, the differential includes the quadratic covariation of the jumps:
$$
d(C\overline G_{\cH})_s=C(s-)d\overline G_{\cH}(s)+\overline G_{\cH}(s-)dC(s)+\Delta C(s)\Delta\overline G_{\cH}(s).
$$
Since $\overline{G}_{\cH}(s)=1-G_{\cH}(s)$, we have $d\overline{G}_{\cH}(s)=-dG_{\cH}(s)$ and $\Delta\overline{G}_{\cH}(s)=-\Delta G{\cH}(s)$. 
Grouping the covariation term with the integrator yields a mathematically exact absorption:
$$
C(s-)dG_{\cH}+\Delta C(s)\Delta G_{\cH}(s)=C(s)dG_{\cH}(s).
$$
This identity holds globally as an equality of measures, since the continuous part of $G_{\cH}$ does not distinguish between the left-hand limit $C(s-)$ and the right-continuous function $C(s)$. 
Consequently, utilizing the right-continuous process elegantly absorbs the jump covariation, reducing the differential to:
$$
d(C\overline G_{\cH})_s=-C(s)d G_{\cH}(s)+\overline G_{\cH}(s-)dC(s).
$$

%%%%%%%%%%%%%%%%%%%%%%%%%%%%%%%%%%%%%%%%%%%%%%%
Substituting this expansion back into the left-hand side of the equation, the terms $-C(s)dG_{\cH}(s)$ and $+C(s)dG_{\cH}(s)$ cancel each other out identically. This simplification yields the differential equation for $C$:
$$
\overline{G}_{\cH}(s-)dC(s)=d\left[F(s)\overline{G}_{\cH}(s)\right] + H(s) dG_{\cH}(s).
$$
According to Lemma \ref{lem: Jeulin}, the left-hand limit of the conditional survival probability is strictly positive almost surely, i.e., $\overline{G}_{\cH}(s-)>0$ for all $s\le\gamma$. Thus, dividing both sides of the equation by $\overline{G}_{\cH}(s-)$, we obtain the explicit differential form for the process $C$:
$$
dC(s)=\frac{d\left[F(s)\overline{G}_{\cH}(s)\right]+H(s)dG_{\cH}(s)}{\overline{G}_{\cH}(s-)}.
$$
Integrating this differential equation over the interval $(0,t]$ yields the integral representation:
$$
C(t)=C(0)+\int_{(0, t]}\frac{d\left[F(s)\overline{G}_{\cH}(s)\right]+H(s)dG_{\cH}(s)}{\overline{G}_{\cH}(s-)}.
$$
By the standard definition of the predictable compensator, it must start at zero, i.e., $A_0=0$. Given the structural representation $A_t=C(t\wedge\gamma)$, this initial condition strictly implies $C(0)$ = 0. Substituting this value into the integral representation completes the proof of formula \ref{thm: compensator: C(t)} for all $t<t_G$.

\ref{thm: compensator: C(t_G)} We now consider the atomic case at the right endpoint, assuming $\PP(\gamma=t_G)>0$. According to the local martingale criterion (Theorem \ref{thm: mart and loc mart}.\ref{thm: mart and local mart: local mart}), the process must satisfy the boundary condition for the jump at time $t_G$:
$$
\EE[\Delta M_{t_G}\mid\cH]=0.
$$
To evaluate this expectation, we partition the sample space into two disjoint sets: $\{\gamma<t_G\}$ and $\{\gamma=t_G\}$. On the set $\{\gamma<t_G\}$, the process $M$ is stopped strictly prior to $t_G$, implying $M_{t_G}=M_{t_G-}$, and thus the jump $\Delta M_{t_G}$ is identically zero. Consequently, the conditional expectation reduces to the evaluation over the atomic set:
$$
\EE\big[\Delta M_{t_G}\ind_{\{\gamma=t_G\}}\mid\cH\big]=0.
$$

On the set $\{\gamma=t_G\}$, the jump of the compensated process is given by the difference between its terminal post-jump value $L^M$ and its pre-jump left-hand limit $F^M(t_G-)$. Substituting this difference into the expectation yields:
$$
\EE\big[(L^M-F^M(t_G-))\ind_{\{\gamma=t_G\}}\mid\cH\big]=0.
$$
By definition of the compensated process, $L^M=L-C(t_G)$ and $F^M(t_G-)=F(t_G-)-C(t_G-)$. Their difference isolates the jump of the compensator:
$$
L^M-F^M(t_G-)=L-F(t_G-)-\Delta C(t_G).
$$
Since $F(t_G-)$ and $\Delta C(t_G)$ are $\cH$-measurable random variables, the properties of conditional expectation directly yield the equation:
$$
\EE[L\ind_{\{\gamma=t_G\}}\mid\cH]-F(t_G-)\PP(\gamma=t_G\mid\cH)-\Delta C(t_G)\PP(\gamma=t_G\mid\cH)=0.
$$
\begin{lem}\label{lem: P(gamma=t_G)>0}
    If $P(\gamma=t_G)>0$, then $P(\gamma=t_G\mid\cH)>0$ almost surely.
\end{lem}
\begin{proof}
    Define the set $B\in\cH$ as $B=\{\omega\in\Omega:\PP(\gamma=t_G\mid\cH)=0\}$. 
    Utilizing the properties of conditional expectation, the joint probability of the jump occurring at time $t_G$ and the event $B$ is evaluated as:
    $$
    \PP(\{\gamma = t_G\}\cap B)=\EE[\ind_{\{\gamma = t_G\}}\ind_B]=\EE\left[\EE[\ind_{\{\gamma =t_G\}}\mid\cH]\ind_B\right]
    $$
    By definition, $\EE[\ind_{\{\gamma=t_G\}}\mid\cH]=\PP(\gamma=t_G\mid\cH)$. On the set $B$, this conditional probability is identically zero. Therefore:
    $$
    \PP(\{\gamma=t_G\}\cap B)=\EE[0\cdot\ind_B]=0.
    $$
    Thus, under the assumption $\PP(\gamma=t_G)>0$, the strict inequality $\PP(\gamma=t_G\mid\cH)>0$ holds almost surely.
\end{proof}

By Lemma \ref{lem: P(gamma=t_G)>0}, division of the equation by $\PP(\gamma=t_G\mid\cH)$ is well-defined. Isolating $\Delta C(t_G)$ yields:
$$
\Delta C(t_G)=\frac{\EE[L\ind_{\{\gamma=t_G\}}\mid\cH]}{\PP(\gamma=t_G\mid\cH)}-F(t_G-).
$$
By Bayes' formula for conditional expectations, the fractional term is identically $\EE[L\mid\cH,\gamma=t_G]$, which by our introduced notation equals $H(t_G)$. Substituting this reduces the expression to:
$$
\Delta C(t_G)=H(t_G)-F(t_G-).
$$
Thus, $C(t_G)=C(t_G-)+\Delta C(t_G)$, which completes the proof of Theorem \ref{thm: compensator}.

To strictly complete the proof, we must verify the predictability of the constructed compensator $A$. 
According to the predictability criterion established in Section \ref{subsec: preliminaries and base criteria}, the stopped process $A_t=C(t\wedge\gamma)$ is $\FF$-predictable if and only if the underlying function $C(t)$ is $\cB(\RR_+)\otimes\cH$-measurable.
Inspecting the integral representation in \eqref{eq: C} and the atomic boundary condition in \eqref{eq: C(t_G)}, the construction of $C(t)$ relies exclusively on the pre-jump process $F$, the conditional jump expectation $H$, and the conditional distribution functions $G_{\cH}$ and $\overline{G}_{\cH}$. 
By the adaptedness of $X$, the component $F$ is inherently $\cH$-measurable.
Similarly, the fundamental properties of the conditional expectation ensure that $H$, $G_{\cH}$, and $\overline{G}_{\cH}$ are strictly $\cH$-measurable. 
Since Lebesgue-Stieltjes integration with respect to an $\cH$-measurable measure rigorously preserves measurability, the resulting function $C(t)$ remains globally $\cH$-measurable on $\RR_+$. 
Consequently, the constructed process A satisfies the strict predictability requirement, validating the Doob-Meyer decomposition and completing the proof of Theorem \ref{thm: compensator}.
\end{proof}

\section{Sigma-Martingales and Localization}\label{sec: s-mar}

The explicit form of the predictable compensator derived in Theorem \ref{thm: compensator} requires the integrability of $F$ and $L$. If $\EE[|L|]=\infty$, the classical Doob-Meyer decomposition does not exist. This section extends the construction of the compensator to processes with heavy-tailed jumps using the framework of $\sigma$-martingales.

\begin{rem}
    The extension to $\sigma$-martingales preserves the canonical representation $X_t=F(t)\ind_{\{t<\gamma\}}+L\ind_{\{t\ge\gamma\}}$. Under the filtration $\cF_t$, any $\sigma$-martingale $X$ is progressively measurable, and its components $F$ and $L$ satisfy the standard measurability conditions with respect to $\cH$ and $\cH\otimes\cB(\overline{\RR}_+)$ respectively.
\end{rem}

\subsection{The Sigma-Martingale Criterion}\label{sec: s-mar: criterion}

To formalize the localization of heavy-tailed jumps, we consider the stochastic integral $Y=\varphi\cdot X$ with respect to a strictly positive predictable process $\varphi$. By the predictability criterion (see \cite[Proposition 2.5]{Zhunussova2026}), the process $\varphi$ is uniquely characterized by an $\cH$-measurable function $\varphi(t)$ on the set $\{t\le\gamma\}$.

\begin{thm}[$\sigma$-Martingale Criterion]\label{thm: sigma-martingale}
    Let $X$ be an adapted right-continuous process of finite variation with the canonical representation $X_t=F(t)\ind_{\{t<\gamma\}}+L\ind_{\{t\ge\gamma\}}$. The process $X$ is a $\sigma$-martingale in the filtration $\FF$ if and only if there exists a strictly positive predictable process $\varphi$, defined by an $\cH$-measurable function $\varphi(t)$, such that the following two conditions hold:
    \begin{enumerate}[label=(\roman*)]
        \item\label{thm: sigma-martingale: varphi} Local integrability of the localized jump:
        \begin{equation}\label{eq: sigma-martingale: varphi}
            \EE\left[\varphi(\gamma)|L-F(\gamma-)|\right]<\infty.
        \end{equation}
        \item\label{thm: sigma-martingale: drift} The drift balance conditions: on the interval $s<t_G$, the differential equation holds,
        \begin{equation}\label{eq: sigma-martingale: drift}
            \overline{G}_{\cH}(s-)dF(s)+F(s)dG_{\cH}(s)+H(s)dG_{\cH}(s)=0
        \end{equation}
        and at the terminal atom $t_G$, the boundary condition holds,
        \begin{equation}\label{eq: sigma-martingale: atom}
            H(t_G)=F(t_G-).
        \end{equation}
    \end{enumerate}
\end{thm}

\begin{proof}
\textbf{($\implies$) Necessity.}
%%%%%%%%%%%%%%%%%%%%%%%%%%%%%%%%%%%%%%%%%%%%%%%

Assume $X$ is a $\sigma$-martingale. 
By definition \cite{JacodShiryaev2003}, there exists a strictly positive $\FF$-predictable process $\varphi$ such that $Y=\varphi\cdot X$ is a local martingale. 
By \cite[Proposition 2.5]{Zhunussova2026}, $\varphi$ admits the representation $\varphi_t=\phi(t)$ on $\{t\le\gamma\}$, where $\phi$ is a strictly positive, $\cB(\RR_+)\otimes\cH$-measurable function.

Since $X$ is stopped at $\gamma$ (i.e., $X_t=X_{t\wedge\gamma}$), the integral process $Y$ is inherently stopped at $\gamma$, yielding $Y_t=Y_\gamma$ on $\{t\ge\gamma\}$. 
On the pre-jump set $\{t<\gamma\}$, the integrability of $Y_t=\int_{(0,t]}\phi(s)dF(s)$ follows directly from the assumption $\EE[\sup\limits_{s\in[0,t_G)}|F(s)|]<\infty$. 
At the jump time $t=\gamma$, we have $\Delta Y_\gamma=\phi(\gamma)\Delta X_\gamma=\phi(\gamma)(L-F(\gamma-))$. This establishes \ref{thm: sigma-martingale: varphi}.
%%%%%%%%%%%%%%%%%%%%%%%%%%%%%%%%%%%%%%%%%%%%%%%

Since $Y$ is a local martingale of finite variation, its predictable compensator $C^Y$ vanishes identically, implying $dC^Y(s)=0$ for all $s<t_G$. 
Applying Theorem \ref{thm: compensator} to the localized process Y, we obtain:
$$
dC^Y(s)=\frac{d\left[F^Y(s)\overline{G}_{\cH}(s)\right]+H^Y(s)dG_{\cH}(s)}{\overline{G}_{\cH}(s-)}=0.
$$
The denominator is strictly positive on $[0,t_G)$, hence the numerator must be zero. 
By definition, $H^Y(s)=\EE[L^Y\mid\mathcal{H},\gamma=s]$. 
By predictability, $\varphi(\gamma)=\varphi(s)$ on the set $\{\gamma=s\}$, which directly yields $H^Y(s)=\varphi(s) H(s)$.
Substituting $dF^Y(s)=\varphi(s)dF(s)$ and expanding the differential $d\left[F^Y(s)\overline{G}_{\cH}(s)\right]$, the multiplier $\varphi(s)$ factors out of the entire nullified numerator. 
Since $\varphi(s)>0$, dividing by it yields the drift balance equation \eqref{eq: sigma-martingale: drift} for the interval $s<t_G$.

At the right endpoint $t_G$, the martingale property of $Y$ implies $\Delta C^Y(t_G)=0$, which is equivalent to $\EE[\Delta Y_{t_G}\mid\cH]=0$. 
On the atomic set $\{\gamma=t_G\}$, the jump is given by $\Delta Y_{t_G}=\varphi(t_G)(L-F(t_G-))$. Given the predictability and strict positivity of $\varphi(t_G)$, its factorization from the conditional expectation directly yields: $\EE[L\mid\cH,\gamma=t_G]=F(t_G-)$. Applying the notation $H(t_G)=\EE[L\mid\cH,\gamma=t_G]$, this strictly reduces to the boundary condition $H(t_G)=F(t_G-)$, completing the proof of necessity.

\textbf{$(\impliedby)$ Sufficiency}.
%%%%%%%%%%%%%%%%%%%%%%%%%%%%%%%%%%%%%%%%%%%%%%%
Assume that for a strictly positive predictable process $\varphi$, conditions \ref{thm: sigma-martingale: varphi} and \ref{thm: sigma-martingale: drift} hold. 
Define the localized process $Y=\varphi\cdot X$. 
Since $X$ is an adapted, right-continuous process of finite variation, $Y$ strictly inherits these properties.
Condition \ref{thm: sigma-martingale: varphi} ensures the local integrability of the localized terminal jump $\Delta Y_\gamma=\varphi(\gamma)(L-F(\gamma-))$.
Combined with the integrability condition $\EE[\sup\limits_{t\in[0,t_G)}|F(t)|]<\infty$ for the pre-jump component, the process $Y$ strictly constitutes a special semimartingale. 
This structural property rigorously justifies the existence of a unique predictable compensator via the Doob-Meyer decomposition. 
By the fundamental results of the general theory of stochastic processes (see Jacod and Shiryaev \cite{JacodShiryaev2003}), $Y$ admits a unique predictable compensator $C^Y$, providing the rigorous framework required to apply Theorem \ref{thm: compensator}.

%%%%%%%%%%%%%%%%%%%%%%%%%%%%%%%%%%%%%%%%%%%%%%%
By condition \ref{thm: sigma-martingale: drift}, the drift balance equation holds on the interval $s<t_G$:
$$
\overline{G}_{\cH}(s-) dF(s) + F(s) dG_{\cH}(s) + H(s) dG_{\cH}(s) = 0.
$$
By the properties of the Lebesgue-Stieltjes integral, multiplying by $\varphi(s)$ yields $dF^Y(s)=\varphi(s) dF(s)$.
Furthermore, conditioned on $\{\gamma=s\}$, the predictability ensures $\varphi(\gamma)=\varphi(s)$, allowing the reconstruction of the localized component $H^Y(s)=\varphi(s)H(s)$. 
Consequently, the multiplied equation takes the form:
$$
\overline{G}_{\cH}(s-) dF^Y(s) + F^Y(s) dG_{\cH}(s) + H^Y(s) dG_{\cH}(s) = 0.
$$
The left-hand side of this equality exactly constitutes the numerator of the predictable compensator differential $dC^Y(s)$ derived in Theorem \ref{thm: compensator} for the localized process $Y$ (incorporating the expansion of the differential $d\left[F^Y(s)\overline{G}_{\cH}(s)\right])$. 
Since the denominator $\overline{G}_{\cH}(s-)$ is strictly positive on $[0,t_G)$, the vanishing numerator identically implies $dC^Y(s)=0$ for all $s<t_G$.

By condition \ref{thm: sigma-martingale: drift}, the boundary equality $H(t_G)=F(t_G-)$ holds at the terminal atom. 
Multiplying both sides by the $\cH$-measurable, strictly positive variable $\varphi(t_G)$ yields $\varphi(t_G)H(t_G)=\varphi(t_G)F(t_G-)$. 
The predictability of $\varphi$ ensures that this equation strictly corresponds to the identity $H^Y(t_G)=F^Y(t_G-)$. 
According to Theorem \ref{thm: compensator}, the jump of the predictable compensator at the terminal atom is given explicitly by $\Delta C^Y(t_G)=H^Y(t_G)-F^Y(t_G-)$. 
Thus, $\Delta C^Y(t_G)=0$. 
Combining this with the initial condition $C^Y_0=0$ and the vanishing differential $dC^Y(s)=0$ on $[0,t_G)$, the predictable compensator vanishes identically, yielding $C^Y \equiv 0$. 
Consequently, the compensated process $Y - C^Y$ coincides with $Y$, proving that the localized process $Y$ is a local martingale. 
By definition, this establishes that the original process $X$ is a $\sigma$-martingale, completing the proof of sufficiency. 
\end{proof}

\subsection{Explicit Predictable Compensator for the Localized Process}

The absence of local integrability for the original $\sigma$-martingale X precludes the existence of the classical Doob-Meyer decomposition. Nevertheless, the localization procedure facilitates the explicit derivation of the predictable compensator for the localized process $Y=\varphi\cdot X$.

By the properties of the stochastic integral, the process Y inherits the canonical single-jump representation:
$$
Y_t=F^Y(t)\ind_{\{t<\gamma\}}+L^Y\ind_{\{t\ge\gamma\}}.
$$

On the set $\{t<\gamma\}$, the pre-jump component is strictly defined by the Lebesgue-Stieltjes integral:

$$
F^Y(t)=\int_{(0,t]}\varphi(s)dF(s).
$$
At the exact instant of the jump $t=\gamma$, the magnitude of the discontinuity is given by $\Delta Y_\gamma=\varphi(\gamma)\left(L-F(\gamma-)\right)$, which completely determines the terminal random variable:
$$
L^Y=F^Y(\gamma-)+\varphi(\gamma)\left(L-F(\gamma-)\right).
$$

Condition \ref{thm: sigma-martingale: varphi} of Theorem \ref{thm: sigma-martingale} ensures the local integrability of the localized terminal jump, implying $\EE[|L^Y|]<\infty$. 
Coupled with the bounded variation of $F^Y$ on finite subintervals of $[0,t_G)$, this guarantees that the localized process $Y$ is a special semimartingale.
Consequently, Theorem \ref{thm: compensator} is directly applicable. 
Substituting the modified components $F^Y$ and $L^Y$ into the compensator equations \eqref{eq: C} and \eqref{eq: C(t_G)} yields the explicit predictable compensator $A^Y_t=C^Y(t\wedge\gamma)$ for the localized process $Y$.

\subsection{Example: Computation of the Predictable Compensator for a Process with a Cauchy Jump}

Assume $t_G=\infty$, and consider a stochastic model where the unconditional distribution of the jump time $\gamma$ is exponential with parameter $\lambda>0$, i.e., $G(t)=1-e^{-\lambda t}$. 
Assume that the initial information $\cH$ contains full knowledge of the terminal jump magnitude, setting $\cH=\sigma(L)$, where the random variable $L$ follows a Cauchy distribution and is independent of $\gamma$. 
Due to the independence of $\gamma$ and $\cH$, the conditional distribution function of the jump time almost surely coincides with the unconditional one:
$$
G_{\cH}(t)=\PP(\gamma\le t\mid\cH)=G(t)=1-e^{-\lambda t}.
$$
Define the original process $X$ with a linear deterministic drift prior to the jump:
$$
X_t=ct\ind_{\{t<\gamma\}}+L\ind_{\{t\ge\gamma\}},
$$
where $c$ is a constant and $F(t)=ct$. 
Since the Cauchy distribution lacks a finite first moment ($\EE[|L|]=\infty$), the process $X$ lacks local integrability, precluding the existence of its classical Doob-Meyer predictable compensator.

To construct the localized process $Y=\varphi\cdot X$, we select the strictly positive, $\cH$-measurable (and consequently predictable) function:
$$
\varphi(s)=\frac{1}{1 + |L|}.
$$

The components of the localized process $Y$ take the explicit form:
$$
F^Y(t)=\int_0^t\varphi(s)cds=\frac{ct}{1+|L|},\quad L^Y=F^Y(\gamma-)+\varphi(\gamma)(L-c\gamma)=\frac{L}{1+|L|}.
$$

Since $|L^Y|<1$, the localized terminal jump is integrable, satisfying condition \ref{thm: sigma-martingale: varphi} of Theorem \ref{thm: sigma-martingale} and ensuring that Y is a special semimartingale of finite variation.

We compute the predictable compensator differential $dC^Y(s)$ according to the criterion in Section \ref{sec: s-mar: criterion}. 
Given the distribution of $\gamma$, we have $\overline{G}_{\cH}(s)=e^{-\lambda s}$ and $dG_{\cH}(s)=\lambda e^{-\lambda s}ds$. 
Due to the $\cH$-measurability of $L$, the conditional expectation on the atom is $H^Y(s)=\EE[L^Y\mid\cH,\gamma=s]=\frac{L}{1+|L|}$.

Substituting these expressions into the numerator of the baseline differential equation (Theorem \ref{thm: compensator}) yields:
$$
d\left[F^Y(s)\overline{G}_{\cH}(s)\right]+H^Y(s)dG_{\cH}(s)=d\left(\frac{cs e^{-\lambda s}}{1+|L|}\right)+\frac{L}{1+|L|}\lambda e^{-\lambda s} ds.
$$

Expanding the differential of the product and grouping the terms by the measure $ds$, we find:
$$
e^{-\lambda s}\frac{c-\lambda cs+\lambda L}{1+|L|}ds.
$$

Dividing by the denominator $\overline{G}_{\cH}(s-)=e^{-\lambda s}$ yields the exact analytical expression for the compensator differential on the interval $s<\gamma$:
$$
dC^Y(s)=\frac{c-\lambda cs+\lambda L}{1+|L|}ds.
$$

Integrating this differential equation over the interval $(0,t\wedge\gamma]$ yields:
$$
A^Y_t=C^Y(t\wedge\gamma)=\frac{c(t\wedge\gamma)-\frac{1}{2}\lambda c(t\wedge\gamma)^2+\lambda L(t\wedge\gamma)}{1+|L|}.
$$

\section{Compliance with Ethical Standards}

\noindent \textbf{Funding:} The author declares that no funds, grants, or other support were received during the preparation of this manuscript.\\
\noindent \textbf{Disclosure of potential conflicts of interest:} The author has no relevant financial or non-financial interests to disclose.\\
\noindent \textbf{Authors Contribution:} Not applicable. The manuscript was written entirely by the sole author. \\
\noindent \textbf{Data Availability Statement:} Data sharing is not applicable to this article as no datasets were generated or analyzed during the current study.

\section{Acknowledgments}
The author wishes to express her sincere gratitude to her scientific advisor, Alexander Alexandrovich Gushchin, for suggesting the problem framework, continuous support, and insightful mathematical discussions that helped shape this work.

\end{document}